\documentclass[11pt]{article}
\usepackage{amssymb,amsmath}
\usepackage{float}
\usepackage{color,fullpage}
\usepackage{cite}
\usepackage[title,titletoc]{appendix}
\usepackage{algorithm,algorithmic}
\usepackage{amsthm}
\newcommand{\RR}{\mathbb{R}}

\newcommand{\BB}{\mathbb{B}}
\newcommand{\EE}{\mathbb{E}}
\newcommand{\dom}{{\mathrm{dom}}\,} 
\newcommand{\ri}{{\mathrm{ri}}}

\newcommand{\prox}{{\mathbf{prox}}}
\newcommand{\cA}{{\mathcal{A}}}

\newcommand{\cY}{{\mathcal{Y}}}
\newcommand{\cP}{{\mathcal{P}}}

\newcommand{\cS}{{\mathcal{S}}}
\newcommand{\cL}{{\mathcal{L}}}
\newcommand{\cE}{{\mathcal{E}}}

\DeclareMathOperator*{\Min}{minimize\quad}
\DeclareMathOperator*{\Max}{maximize\quad}
\newcommand{\st}{\mbox{subject to}}
\newtheorem{assumption}{Assumption}[section]

\newtheorem{theorem}{Theorem}[section]
\newtheorem{lemma}{Lemma}[section]
\newtheorem{definition}{Definition}[section]

\newtheorem{proposition}{Proposition}[section]

\makeatletter

\@addtoreset{equation}{section} \makeatother

\begin{document}

\title{Linear convergence of random dual coordinate incremental aggregated gradient methods}

\author{Hui Zhang\thanks{
Department of Mathematics, National University of Defense Technology,
Changsha, Hunan 410073, China.  Email: \texttt{h.zhang1984@163.com}}
\and Yu-Hong Dai\thanks{100190 $\&$ School of Mathematical Sciences,
Chinese Academy of Sciences, Beijing 100049, China, Email: \texttt{dyh@lsec.cc.ac.cn}
}
\and Lei Guo\thanks{Corresponding author. School of Business, East China University of Science and Technology, Shanghai 200237, China. Email: \texttt{lguo@ecust.edu.cn}
}
}

\date{\today}

\maketitle

\begin{abstract}
In this paper, we consider the dual formulation of minimizing
$
\sum_{i\in I}f_i(x_i)+\sum_{j\in J} g_j(\mathcal{A}_jx)
$
with the index sets $I$ and $J$ being large. To address the difficulties from the high dimension of the variable $x$ (i.e., $I$ is large) and the large number of component functions $g_j$ (i.e., $J$ is large), we propose a hybrid method called the random dual coordinate incremental aggregated gradient method by blending the random dual block coordinate descent method and the proximal incremental aggregated gradient method. To the best of our knowledge, no research is done to address the two difficulties simultaneously in this way. Based on a newly established descent-type lemma, we show that linear convergence of the classical proximal gradient method under error bound conditions could be kept even one uses delayed gradient information and randomly updates coordinate blocks. Three application examples are presented to demonstrate the prospect of the proposed method.
\end{abstract}




\section{Introduction}
The following structured composition convex optimization has been well studied in the literature
\begin{equation}\label{p}
\Min_{x\in \cE_1} F(x):= f(x)+g(\cA x), \\[4pt]
\end{equation}
where $f,g$ are proper closed convex functions, $\cE_1$ a Euclidean space, $\cA$  a given linear mapping. This class of problems frequently appears in many fields such as mathematical optimization, signal/imaging processing, machine learning and big data. Many efficient numerical algorithms for solving it are proposed in the literature. Among first-order methods, the proximal gradient (PG) method may be the most well-known. A standard assumption, required by the PG method, is that $g$ is gradient-Lipschitz-continuous and the proximal operator of $f$ can be easily computed. However, in many cases this assumption fails to hold for the primal problem \eqref{p} but fortunately it holds for its Fenchel-Rockafellar dual problem
\begin{equation}\label{d}
\Min_{y\in \cE_2} D(y):= f^*(-\cA^*y)+g^*(y),
\end{equation}
where $\cE_2$ is the dual space to $\cE_1$, and $f^*$ and $g^*$ are conjugate functions of $f$ and $g$ respectively. That is to say, $f^*$ may be gradient-Lipschitz-continuous and the proximal operator of $g^*$ can be easily computed although $g$ is not gradient-Lipschitz-continuous or the proximal operator of $f$ cannot be easily computed. The total variation denoising problem, formulated as $\Min \frac{1}{2}\|x-v\|^2+ \lambda \|Dx\|_1$ with given signal $v$, regularization parameter $\lambda>0$, and matrix $D$, is such an example. Therefore, the PG method could be directly applied to the dual problem \eqref{d}, and then the dual-based PG methods follows; see e.g. Beck's recent book \cite{beck2017first} for sublinearly convergent dual PG methods and paper \cite{Necoara2019Linear} for linearly convergent random dual block coordinate descent methods. However, these existing dual-based methods only exploit the separability of $g$ and then may be not suitable for solving problems where $f$ has a huge number of component functions. This motivates us to exploit the separability of $f$ and $g$ to design efficient methods.

In this paper, we consider problem \eqref{p} where both $f$ and $g$ have separable structure as follows
\begin{equation*}\label{pf}
f(x):=\sum_{i\in I} f_i(x_i),\ g(\cA x):=\sum_{j\in J} g_j(\cA_jx),
\end{equation*}
where $f_i,g_j$ are all proper closed convex functions, $\cA_j$ given linear mappings, and $I,J$ index sets. Specifically, this paper focuses on designing efficient methods for solving
\begin{equation}\label{p3}
\Min_{x\in \cE_1} F(x) = \sum_{i\in I}f_i(x_i)+\sum_{j\in J} g_j(\cA_jx),
\end{equation}
with $|I|$ and $|J|$ being large,  whose Fenchel-Rockafellar dual problem can be written as
\begin{equation}\label{d3}
\Min_{y\in \cE_2} D(y):=  \sum_{i\in I}f_i^*(-\sum_{j\in J}\cA_{ji}^*y_j) +  \sum_{j\in J}g_j^*(y_j),
\end{equation}
where the linear mappings $\cA_{ji}$ will be defined in Section 2. The cardinality $|I|$ is the dimension of the primal variable $x$ or the number of component functions $f_i^*$ in the dual problem, and the cardinality $|J|$ is the dimension of the dual variable $y$ or the number of component functions $g_j$ in the primal problem \eqref{p3}. When the PG method is applied to the dual problem, large $|I|$  implies a huge gradient computation complexity and large $|J|$ implies a large proximal point computation complexity at each step.  Therefore, the case of $|J|$ being large leads to the development of block-type PG variants such as the random dual block coordinate descent (DBCD) method \cite{Peter2014,Necoara2019Linear}; while the case of $|I|$ being large leads to the development of incremental-type methods such as the proximal incremental aggregated gradient (PIAG) method \cite{Gurbuzbalaban2015on,Vanli2016}.
To the best of our knowledge, in the literature there is no research that deals with the case where both $|I|$ and $|J|$ are large. As a first try, we propose a hybrid method, called the random dual coordinate incremental aggregated gradient (abbreviated by RDCIAG) method, by blending the random DBCD method and the PIAG method. At the algorithmic level, our proposed method could be viewed as a further research for the random DBCD method and the PIAG method.

In analyzing the linear convergence rate of iterate methods, error bound conditions have been shown to be extremely useful  \cite{luo1993error,Bolte2015From}. Global error bound conditions may be too stringent in practice, which will substantially restrict the applicability, and local error bound conditions are not sufficient to ensure linear convergence for non-monotone iterative methods.
In this paper, we will use the bounded error bound condition to analyze the iteration complexity of the RDCIAG method, which is non-monotone. The bounded error bound condition is actually the bounded metric subregularity of the subdifferential of the dual objective function in problem \eqref{d3}. Many sufficient conditions for ensuring bounded error bound condition to hold are given in \cite{Drusvyatskiy2018Error} and \cite{Ye2018} (see Section \ref{sec:error}). Based on the bounded error bound condition, we show that the RDCIAG method converges linearly. The proof depends on two pillars: One is the tail-vanishing lemma introduced in \cite{Aytekin2016}, and the other is a newly developed descent-type lemma, which delicately combines the random block coordinate descent and the PIAG descent.

The remainder of the paper is organized as follows. In Section 2, we present the basic notation and some elementary preliminaries, and  existing sufficient conditions for ensuring bounded error bound condition to hold. In Section 3, we propose the RDCIAG method by blending the random DBCD method and the PIAG method. In Section 4, we study the linear convergence of the proposed method. In Section 5, we present three application examples to demonstrate the prospect of the proposed method. Finally, section 6 gives some concluding remarks.

\section{Preliminaries and preliminary results}

In this paper, we restrict our analysis in finite dimensional Euclidean spaces. Let $\cE$ be a Euclidean space and $\|\cdot\|$ the associated Euclidean norm. For a closed subset $Q\subseteq \cE$ and a point $x\in \cE$, we define by
$d(x,Q):=\inf_{y\in Q}\|x-y\|$ the distance function from $x$ to $Q$ and by $\cP_Q(x):=\{y\in Q: \|y-x\| = d(x,Q)\}$ the set of projection from $x$ to $Q$.
The closed ball around $x\in\cE$ with radius $r>0$ is denoted by $\BB_\cE(x,r):=\{y\in\cE:\|x-y\|\leq r\}$. If the central point is zero and the around space $\cE$ is known, we abbreviate the closed ball with radius $r$ as $\BB_r$. We let ``int" and ``ri" denote the interior and relative interior of a given set respectively.

Given $m$ Euclidean spaces $\{\cE_i: i=1,\cdots, m \}$  with inner products $\langle \cdot, \cdot\rangle_{\cE_i}$, their Cartesian product, defined by
$$ \cE:=\bigoplus_{i=1}^m\cE_i=\{(x_1,x_2,\cdots, x_m): x_i\in \cE_i, i=1,\cdots, m\},$$
is a Euclidean space equipped with the component-wise addition and the scalar-vector multiplication.
The inner product in $\bigoplus_{i=1}^m\cE_i$ is defined as
$$\big\langle (x_i)_{i=1}^m, (z_i)_{i=1}^m \big\rangle_\cE:= \sum_{i=1}^m \langle x_i, z_i\rangle_{\cE_i}.$$
This paper focuses on two Euclidean spaces $\cE_1$ and $\cE_2$ which are defined as Cartesian products of a group of Euclidean spaces:
\[
\cE_1 := \bigoplus_{i\in I} \cE_{1,i},\quad \cE_2 := \bigoplus_{j\in J }\cE_{2,j},
\]
where $I$ and $J$ are two finite index sets.  When no confusion arises, we will omit the subscript.
A linear transform $\cA$ is defined from $\cE_1$ to $\cE_2$ as follows:
\[
\cA x= (\cA_jx)_{j\in J}=\Big(\sum_{i\in I}\cA_{ji}x_i\Big)_{j\in J},
 \]
 where $\cA_{ji}$ are linear transforms from $ \cE_{1,i}\rightarrow \cE_{2,j}$. The associated adjoint transform $\cA^*:\cE_2\to \cE_1$ is defined by
$$\cA^* y= \sum_{j\in J}\cA_{j}^*y_j=\Big(\sum_{j\in J}\cA_{ji}^*y_j\Big)_{i\in I}.$$
The norm of the linear transform $\cA$ is defined by
$$\|\cA\|:=\max\{\|\cA x\|_{\cE_2}: \|x\|_{\cE_1}\leq 1\}.$$

Let $\Gamma_0(\cE)$ denote the class of proper and lower semicontinuous convex functions from $\cE$ to $(-\infty, +\infty]$.
Let $\phi_i\in \Gamma_0 (\cE_{1,i}), i\in I$ and $\phi\in \Gamma_0 (\cE_1)$. We say that $\phi$ is separable if it has the form
$\phi(x)=\sum_{i\in I} \phi_i(x_i)$.

Let $\phi:\cE\rightarrow (-\infty, +\infty]$ be a proper convex function. The effective domain of $\phi$ is defined by $\dom \phi :=\{x\in \cE: \phi(x)< +\infty\}$.
The proximal mapping of $\phi$ is defined by
$$\prox_{\lambda \phi}(x):=\arg\min_{y\in\cE}\{\phi(y)+\frac{1}{2\lambda}\|y-x\|^2\}.$$
The conjugate (also called Fenchel conjugate, or Legendre tansform, or Legendre-Fenchel tansform) of $\phi$ is
$$\phi^*(y)=\sup_{x\in \cE}\{\langle x,y\rangle_\cE-\phi(x)\}.$$
The subdifferential of $\phi$  at $x$ is defined by
$$\partial \phi (x):= \{ y \in \cE: \phi(u)\geq \phi(x)+ \langle y, u-x\rangle_\cE,\quad \forall u\in \cE \}.$$
We say that $\phi$ is subdifferentiable at $x\in \cE$ if $\partial \phi(x)\neq \emptyset$. The elements of $\partial \phi(x)$ are called the subgradients of $\phi$ at $x$.

A closed proper convex function $\phi$ is called essentially smooth if $\partial \phi$ is a single-valued mapping. In this case, $\partial \phi(x)=\nabla \phi(x)$ when $x\in {\rm int}\, \dom \phi$ and $\partial \phi(x)=\emptyset$ otherwise \cite[Theorem 26.1]{rockafellar1970convex}.
$\phi$  is called essentially strictly convex if $\phi$ is strictly convex on every convex subset of $\{x: \partial \phi(x)\not=\emptyset\}$. A closed proper convex function is essentially strictly convex if and only if its conjugate is essentially smooth \cite[Theorem 26.3]{rockafellar1970convex}.

We say that $\phi:\cE\rightarrow (-\infty, +\infty)$ is gradient-Lipschitz-continuous with  modulus $L>0$ if
\begin{equation*}
\|\nabla \phi(x)-\nabla \phi(y)\|\leq L\|x-y\|, \quad\forall x, y\in\cE. \label{Lip}
\end{equation*}
We say that $\phi:\cE\rightarrow (-\infty, +\infty]$ is strongly convex with modulus $\mu>0$ if for any $\alpha\in [0,1]$,
\begin{equation*}
\phi(\alpha x+(1-\alpha)y)\leq \alpha \phi(x)+(1-\alpha)\phi(y)-\frac{1}{2}\mu\alpha(1-\alpha)\|x-y\|^2, \quad\forall x, y\in\dom \phi, \label{SC1}
\end{equation*}
or if (when it is differentiable)
\begin{equation*}
\langle \nabla \phi(x)-\nabla \phi(y), x-y\rangle \geq \mu \|x-y\|^2,\quad\forall x,y\in \dom \phi.\label{SC}
\end{equation*}

A multi-function $\Psi: \mathcal{E}_1\rightrightarrows\mathcal{E}_2$ is a mapping assigning each point in $\cE_1$ to a subset of $\cE_2$. The graph of $\Psi$ is defined by
$$\verb"gph"(\Psi):=\{(u,v)\in\cE_1\times\cE_2: v\in \Psi(u)\}.$$
The inverse map $\Psi^{-1}:\mathcal{E}_2\rightrightarrows\mathcal{E}_1$  is defined by
$$ \Psi^{-1}(v):=\{u\in \cE_1: v\in \Psi(u)\}.$$

The following condition plays an important role in deducing error bound conditions for structured convex optimization problems.

\begin{definition}[\cite{Bauschke1993On}, Bounded linear regularity]
We say that the pair sets $\{A, B\}$ have the bounded linear regularity (BLR) if for every bounded set $C$, there exists a constant $\kappa>0$ such that
$$d(x, A\bigcap B)\leq \kappa (d(x, A)+d(x,B)), \quad \forall x\in C.$$
\end{definition}

We now introduce two weakened conditions for the strong convexity.

\begin{definition}[\cite{Zheng2014Metric}, Bounded metric subregularity]\label{def-bms}
We say that a multi-function $\Psi: \mathcal{E}_1\rightrightarrows\mathcal{E}_2$ is bounded metrically subregular (BMS) at $(\bar{u},\bar{v})\in \verb"gph"(\Psi)$ if for any compact set $U$ with $\bar{u}\in U$, there exists $\kappa>0$ such that
\begin{equation*}\label{BMS}
d(u, \Psi^{-1}(\bar{v}))\leq \kappa d(\bar{v}, \Psi(u)),~~\forall u\in U.
\end{equation*}
\end{definition}

It should be noted that the requirement $\bar{u}\in U$ for $\bar{u}$ in Definition \ref{def-bms} is not necessary by noting the arbitrariness of compact set $U$. Then we can simply say that $\Psi$ is BMS at $\bar{v}$. As noted in \cite{Ye2018}, a polyhedral multi-function must be BMS at every point in the graph of the multi-function. Many functions are polyhedral such as the polyhedral convex function and the convex piecewise linear quadratic function. The norm function $\|\cdot\|$ also satisfies the BMS property; see \cite{Ye2018} for more details and examples.

It is well-known that a convex differentiable function $\psi$ is strongly convex if and only if its conjugate is gradient-Lipschitz-continuous \cite{Hiriart2004}. The following result shows that $\partial\psi$ is BMS at $\bar{u}$ when its inverse (or the subdifferential of the conjugate of $\psi$) is upper Lipschitz continuous at $\bar{u}$. This indicates that the BMS property is a weak version for the strong convexity.

\begin{proposition}\label{pro-bms}
Let $\psi \in \Gamma_0(\cE)$. The multi-function $\partial \psi$ is BMS at $\bar{x}$ if $\partial \psi^{-1}$ (or $\partial \psi^*$) is upper Lipschitz continuous at ${\bar x}$ in the sense that there exist $\delta>0$ and $\kappa>0$ such that
\begin{equation}\label{prop1}
\partial \psi^{-1}(x) \subseteq \partial \psi^{-1}(\bar{x}) + \kappa\|x-\bar{x}\| \BB_1,\quad \forall x\in \BB(\bar{x},\delta).
\end{equation}
\end{proposition}

For a function $\phi \in \Gamma_0(\cE)$, the BMS of $\partial \phi$ is equivalent to the firm convexity of $\phi$ as follows \cite{Drusvyatskiy2018Error}.

\begin{definition}[\cite{Drusvyatskiy2018Error}, Firm convexity]
A closed convex function $\phi$ is firmly convex relative to a vector $v$ if the tilted function $\phi_v(x) := \phi(x) - \langle v, x\rangle$ satisfies
the quadratic growth condition: for any compact set $V$ there is a constant $\sigma$
satisfying
\[
\phi_v(x) \ge {\rm inf} \phi_v + \frac{\sigma}{2} d^2(x,(\partial \phi_v)^{-1}(0))\quad \forall x\in V.
\]
\end{definition}

To ensure the existence of optimal solutions of dual problems and zero duality gap, we make the following standard assumptions throughout this paper.
\begin{assumption}\label{ass}
Let $f \in \Gamma_0(\cE_1), g\in \Gamma_0(\cE_2)$.
We assume that an optimal solution ${\bar x}$ of problem \eqref{p} exists. Moreover, assume one of the following conditions holds:
\begin{itemize}
  \item[(i)]The following non-degenerate condition holds, i.e.,
\begin{equation}\label{cond1}
0\in \ri(\dom g-\cA \dom f).
\end{equation}
  \item[(ii)]If $g$ is a polyhedral function, $\dom g \cap \cA \ri \dom f\not=\emptyset$ holds true.
  \item[(iii)] If both $f$ and $g$ are polyhedral functions, $\dom g \cap \cA \dom f\not=\emptyset$ holds true.
\end{itemize}
\end{assumption}

A sufficient and necessary  condition for ensuring \eqref{cond1} to hold is $\ri \dom g \cap \cA \ri \dom f\not=\emptyset$ \cite[Proposition 15.24]{Bauschke2017}.
When  $f$ and $g$ are separable, it is not hard to verify that
condition \eqref{cond1} can be equivalently written as
\begin{equation*}\label{cond1.3}
 (\ri\dom g_j)\bigcap \left(\sum_{i\in I}\cA_{ji} (\ri\dom f_i)\right)\neq \emptyset,\quad j\in J.
\end{equation*}

The following well-known result is  fundamental to study the primal-dual gap and the relationship between primal and dual solutions.
Let $P^*$ and $D^*$ denote the primal and dual optimal function value respectively, i.e.,
\[P^* = \inf_{x\in \cE_1} f(x)+g(\cA x),\ D^*=\inf_{y\in \cE_2} f^*(-\cA^*y)+g^*(y).
\]

\begin{lemma}[Fenchel-Rockafellar Duality, Theorem 15.23, Fact 15.25, and Theorem 19.1 in \cite{Bauschke2017} ]\label{FRthm}
Let $f \in \Gamma_0(\cE_1), g\in \Gamma_0(\cE_2)$, and $\cA$ be a linear transform from $\cE_1$ to $\cE_2$.
\begin{itemize}
  \item[(i)]
If Assumption \ref{ass} holds, then the duality gap is zero, i.e., $ P = - D^*$, and the dual problem possesses an optimal solution.

\item[(ii)] $ P^* =f(\bar{x})+g(\cA \bar{x}),\ D^*=f^*(-\cA^*\bar{y})+g^*(\bar{y}),$ and $P^*=-D^*$,
if and only if the KKT conditions hold
\begin{equation}\label{KTcond}
-\cA^*\bar{y}\in \partial f(\bar{x}),\ \bar{y}\in \partial g(\cA \bar{x}).
\end{equation}
\end{itemize}
\end{lemma}

Under Assumption \ref{ass}, since ${\bar x}$ is an optimal solution of problem \eqref{p}, by Fermat's rule, we have $ 0\in \partial (f + g \circ \cA) (\bar x)$. Then by Assumption \ref{ass} and \cite[Theorem 16.47]{Bauschke2017}, it follows that $0\in \partial f(\bar x) + \cA^* \partial g (\cA \bar x)$.
Let $\bar y \in \partial g (\cA \bar x)$ such that $-\cA^*\bar{y}\in \partial f(\bar{x})$. The KKT conditions \eqref{KTcond} follows immediately. Thus by Lemma \ref{FRthm}(ii), the duality gap is zero and $\bar y$ is the optimal solution of the dual problem \eqref{d}. We denote $\cY$ as the optimal solution set of problem \eqref{d}. By Lemma \ref{FRthm}(ii), we obtain
\begin{equation}\label{dualset}
\cY=\{\bar{y}\in \cE_2: \bar{y}\in \partial g(\cA\bar{x}), \cA^*\bar{y}\in -\partial f(\bar{x})\}=\partial g(\cA\bar{x})\bigcap (\cA^*)^{-1}(-\partial f(\bar{x})).
\end{equation}

\subsection{Bounded error bound conditions}\label{sec:error}

Error bound condition for the optimal solution set has been shown to be extremely useful in analyzing the linear convergence of iterate methods. Deducing global or local error bounds for mathematical programming as in \eqref{d} has a long history and been extensively investigated; see e.g. \cite{Zhou2015A,Drusvyatskiy2018Error,Ye2018} for general settings and \cite{Lai2012Augmented,Frank2015linear,Bolte2015From,Beck2015Linearly
,Necoara2019Linear,tseng2009coordinate,tseng2010approximation} for some special cases.

This subsection focuses on the {\it bounded} error bound conditions, which lie between global and local error bound conditions, for the optimal solution set $\cY=\{y\in\cE_2: 0\in\partial D(y)\}$ of the dual problem \eqref{d}.
That is, for any compact set $V\subseteq \cE_2$, there exists a constant $\kappa>0$ such that
\begin{equation}\label{err}
d(y,\cY) \le \kappa d(0,\partial D(y)),\quad \forall y\in V.
\end{equation}
It is  easy to see that $\cY = \partial D^{-1}(0)$. Let $0\in \partial D(\bar{y})$ and $\bar{y}\in V$. Then \eqref{err} can be written as
\[
d(y,\partial D^{-1}(0)) \le \kappa d(0,\partial D(y)),\quad \forall y\in V.
\]
This is exactly the BMS for the multi-function $\partial D$ at $(\bar{y},0)$.

Since the global error bound result is stronger than the bounded error bound result, we first collect the known sufficient results for ensuring the global error bound to hold in the following. The first condition (a) follows from \cite[Lemma 1]{gong2014linear} (see also Lemma 2.5 in \cite{Beck2015Linearly}) and the second condition (b) follows from the strong convexity of $f^*$ immediately.
\begin{itemize}
\item[(a)] Assume that the function $g^*$ is the indictor function of a polyhedral set $W$, and $f^*$ is a strongly convex differentiable function with $\nabla f^*$ Lipschitz continuous on $W$.

\item[(b)] Assume that $\cA$ is an identity matrix, and the function $f^*$ is a strongly convex differentiable function with $\nabla f^*$ Lipschitz continuous on $\dom(g^*)$.
\end{itemize}

In a recent paper \cite{Ye2018}, for ensuing the linear convergence of the randomized block coordinate proximal gradient method, Ye et al. investigated the BMS property of $\partial D$ (using our notation of this paper) and gave some sufficient conditions for the BMS to hold by making use of the bounded metric subregular intersection theorem (see \cite[Proposition 9]{Ye2018}).
The following result collects the sufficient conditions for the BMS of $\partial D$ to hold.

\begin{proposition}\label{pro-su}\cite[Theorem 2, Theorem 3]{Ye2018}
Assume that $f^*$ is strongly convex on any convex compact subset of $\dom f^*$, that $f^*$ is continuously differentiable on $\dom f^*$ which is assumed to be open and
$\nabla f^*$ is Lipschitz continuous on any compact subset of $\dom f^*$. Then $\partial D$ is BMS at $(\bar{y},0)$ if one of the following conditions holds
\begin{itemize}
\item[(i)] $\partial g^*$ is a polyhedral multi-function,
\item[(ii)] $\partial g^*$ is BMS at $(\bar{y},\cA \nabla f^* (-\cA^* \bar{y}))$ and the set $\{y: 0\in -\cA \nabla f^* (-\cA^* \bar{y})+\partial g^*(y)\}$ is a convex polyhedral set.
\end{itemize}
\end{proposition}

It is well-known that there is a close relationship between error bound condition and quadratic growth condition. Following the proof technique in \cite[Theorem 4.3]{drusvyatskiy2015quadratic}, it is not hard to verify that the bounded error bound condition \eqref{err} can imply the bounded quadratic growth condition with $(r,\sigma)$, i.e., for any $r>0$ there exists $\sigma>0$ such that
 \begin{equation}\label{QG}
D(y)\ge D^* +  \frac{\sigma}{2} d^2(y,\cY), \quad \forall y\in \BB_r.
 \end{equation}
By using the convexity of $D$, it is also easy to verify that the bounded quadratic growth condition implies the bounded error bound condition \eqref{err}.

Drusvyatskiy and Lewis showed the bounded quadratic growth conditions without assuming the commonly used strong convexity of $f^*$ in \cite[Theorem 4.3]{Drusvyatskiy2018Error}. The main techniques they used are a newly introduced concept called firm convexity and the bounded linear regularity for the following expression that
\begin{equation}\label{solutionset2}
\cY= \partial g(\cA\bar{x})\bigcap (\cA^*)^{-1}(-\partial f(\bar{x})).
\end{equation}

The following results essentially follow from \cite[Theorem 4.3]{Drusvyatskiy2018Error} and \cite[Corollary 4.4]{Drusvyatskiy2018Error}. For the readers' convenience, we give a brief proof in the appendix.

\begin{proposition}\label{prop-1}Let $f$ be an essentially strictly convex function. Then the bounded quadratic growth condition with $(r,\sigma)$ holds if one of the following conditions hold
\begin{itemize}
  \item[(i)] $f_i$ is upper Lipschitz continuous at $\bar{x}_i$ for all $i\in I$, $g_j$ is upper Lipschitz continuous at $\cA_j \bar{x}$ for all $j\in J$, and
\begin{equation}\label{cond2}
0\in \ri(\partial f(\bar{x})+\cA^*\partial g(\cA\bar{x})),
\end{equation}
  \item[(ii)] $\partial f_i$ and $\partial g_j$ are polyhedral for all $i\in I$ and $j\in J$.
\end{itemize}
\end{proposition}

\section{The proposed algorithm}\label{sec:alg}

In this section, we will propose the RDCIAG method based on the certain structure of the dual objective function. Throughout this section, we always assume that  $f$ is strongly convex and hence its conjugate $f^*$ in the dual objective function is gradient-Lipschtiz-continuous, which is not necessarily strongly convex.

\subsection{Warm-up: the random dual block coordinate algorithm}
Recently, the authors of \cite{beck2017first,Necoara2019Linear} introduced (random) dual coordinate descent methods for solving the primal-dual problems \eqref{p3}-\eqref{d3} with $|I|=1$.
Here, we first recall their algorithmic idea. Recall that the dual objective function is $f^*(-\cA^*y)+  \sum_{j\in J}g_j^*(y_j)$ with smooth function $f^*$. Applying the proximal gradient method to the dual problem, we obtain that
\begin{equation}\label{PG1}
y^{k+1}=\arg\min_{y\in \cE_2} \left\{\sum_{j\in J}g_j^*(y_j)+\langle -\cA\nabla f^*(-\cA^*y^k), y-y^k\rangle +\frac{1}{2\alpha^k}\|y-y^k\|^2\right\}.
\end{equation}
In the following we reformulate \eqref{PG1} as a primal-dual scheme by introducing primal variables based on the KKT conditions \eqref{KTcond}.
Let
\begin{equation}\label{primalv}
x^k:=\nabla f^*(-\cA^*y^k)=\nabla f^*( -\sum_{j\in J}\cA_j^*y_j^k);
\end{equation}
then $$-\cA\nabla f^*(-\cA^*y^k)=-\cA x^k=(-\cA_jx^k)_{j\in J}.$$
Note that the objective function in \eqref{PG1} is separable with respect to $y$. In terms of the primal variable $x^k$, the inner product in \eqref{PG1} can be simplified and then the update of dual variables can be rewritten as the following form
\begin{equation}\label{dualv}
y^{k+1}_j=\arg\min_{y_j\in \cE_{2,j}} \left\{ g_j^*(y_j)+\langle -\cA_jx^k, y_j-y_j^k\rangle +\frac{1}{2\alpha^k}\|y_j-y_j^k\|^2\right\},\quad j\in J,
\end{equation}
or equivalently,
\begin{equation}\label{dualv1}
y^{k+1}_j=\prox_{\alpha_kg^*_{j}}(y_{j}^k+\alpha_k\cA_{j}x^k),\quad j\in J.
\end{equation}
Let $y^0\in \cE$ be the initial point. By collecting the primal-dual updates \eqref{primalv}-\eqref{dualv1} and introducing the idea of random coordinate updates, we present the following random primal-dual algorithm: choose uniformly random index $j_k\in J$ and~update
\begin{eqnarray}\label{alg1}
\left\{\begin{array}{lll}
x^k &= &\nabla f^*( -\sum_{j\in J}\cA_j^*y_j^k),    \\
y_{j_k}^{k+1}&=& \prox_{\alpha_kg^*_{j_k}}(y_{j_k}^k+\alpha_k\cA_{j_k}x^k),\\
y_j^{k+1}&=&y_j^k,\quad \forall j\neq j_k.
\end{array} \right.
\end{eqnarray}
The convergence of the dual proximal gradient algorithm \eqref{dualv} and its randomized variant \eqref{alg1} has been extensionally studied; see e.g. \cite{beck2017first,zhang2016new,Peter2014,Necoara2019Linear}.

\subsection{The RDCIAG method}
The algorithm \eqref{alg1} is very general to include the algorithms in \cite{beck2017first,Necoara2019Linear}  as special cases. But it only exploits the separability of $g$. In this subsection, we consider how to further exploit the separability of $f$ in designing algorithms.
Recall that the objective function of the dual problem \eqref{d3} is
$$D(y)=\sum_{i\in I} f_i^*(-\sum_{j\in J}\cA_{ji}^*y_j) +  \sum_{j\in J}g_j^*(y_j).$$
To simplify the notation, we let $h_i(y):=f_i^*(-\sum_{j\in J}\cA_{ji}^*y_j)$ and solve the following problem
$$\Min_{y\in\cE_2} \sum_{i\in I}  h_i(y) +  \sum_{j\in J}g_j^*(y_j).$$
First, let us turn our attention to the first term in the above problem. If the cardinality of the index set $I$ is very large, which happens in big data models and huge-dimensional problems, evaluating the full gradient of $\sum_{i\in I}  h_i(y)$ at some point as done in the algorithm \eqref{alg1} is costly and even prohibitive. On the other hand, in many practical problems such as distributed optimization and network optimization, delay of gradient information update is very common.
To overcome these issues, some practice-driven algorithms are developed such as stochastic gradient-type methods and incremental aggregated gradient-type methods.  Compared with stochastic gradient-type algorithms, incremental aggregated gradient-type methods are much easier to implement in large-scale setting since the latter does not require independent and random sampling in each iteration. Meanwhile, incremental aggregated gradient-type methods usually outperform their stochastic counterparts because they visit each component function at each iterate and update it in a period while in stochastic settings the selection of component functions is at random and some of them may be not visited in each epoch. Besides, in practice incremental methods are broadly employed for a long history in several advanced fields such as  neural networks, reinforcement learning, and optimal control \cite{bertsekas2016nonlinear,bertsekas2019}.  Even though, much less attention is paid to incremental aggregated gradient-type methods possibly due to the difficulty of convergence analysis. During the past few years, some remarkable progresses of determined incremental methods have been made; see e.g. \cite{Bertsekas2015incre,Gurbuzbalaban2015on,Vanli2016,Mert2019Convergence,peng2019,zhang2020}. The study may be viewed as a continuum of these progresses.

By borrowing the ideas of random block coordinate descent algorithms and incremental aggregated gradient methods, we propose the following RDCIAG method: Choose uniformly a random index $j_k\in J$ and update
\begin{eqnarray}\label{alg2}
y_{j_k}^{k+1}=\arg\min\limits_{y_{j_k}\in \cE_{2,{j_k}}}\left\{\Big\langle \sum_{i\in I} \nabla_{j_k} h_i(y^{k-\tau_k^i}), y_{j_k}-y_{j_k}^k\Big\rangle +\frac{1}{2\alpha_k}\|y_{j_k}-y_{j_k}^k\|^2+ g^*_{j_k}(y_{j_k})\right\},
\end{eqnarray}
and let $y_j^{k+1}= y_j^k, \forall j\neq j_k$, or equivalently
\begin{eqnarray}\label{alg3x}
\left\{\begin{array}{lll}
x_i^k &= & \nabla f_i^*( -\sum_{j\in J}\cA_{ji}^*y_j^{k-\tau_k^i}),\quad i\in I   \\
y_{j_k}^{k+1}&=& \prox_{\alpha_kg^*_{j_k}}(y_{j_k}^k+\alpha_k\sum_{i\in I}\cA_{j_ki}x^k_i)\\
y_j^{k+1}&=&y_j^k,\quad \forall j\neq j_k,
\end{array} \right.
\end{eqnarray}
where $\tau_k^i \in [0,\tau]$ for all $k$ and $i$ are delayed indexes and $\tau\ge0$ is the largest delayed factor. The idea of designing dual incremental aggregated methods is not new; see e.g. \cite{A2008Blatt,Bertsekas2015incre}. Our novelty lies in considering the proximal incremental aggregated gradient descent method and the random dual coordinate descent method in a unified way. This unified scheme sufficiently utilizes the separability of both $f$ and $g$, and should be suitable for large-scale problems and distributed problems due to its low computational load in each iteration and its allowance for delayed gradient computation.

In order to analyze the convergence of the RDCIAG method, we introduce an equivalent expression. Define the embedding operator $U_q: \cE_{2,q}\rightarrow \cE_2$ as $U_q y_q=(z_j)_{j\in J}$ with
\begin{eqnarray*}
z_j =
\left\{\begin{array}{lll}
y_q,       &j=q, \\
0,         &\textrm{otherwise},
\end{array} \right.
\end{eqnarray*}
and a determined updated variable $\widetilde{y}^{k+1}$ via
\begin{eqnarray}\label{update1}
\widetilde{y}_{j}^{k+1}=\arg\min\limits_{y_j\in \cE_{2,j}}\left\{\Big\langle \sum_{i\in I} \nabla_j h_i(y^{k-\tau_k^i}), y_j-y_j^k\Big\rangle +\frac{1}{2\alpha_k}\|y_j-y_j^k\|^2+ g^*_{j}(y_j)\right\},\quad j\in J.
\end{eqnarray}
Then the RDCIAG method can be written as
\begin{equation}\label{rpiag}
y^{k+1}=U_{j_k}\widetilde{y}_{j_k}^{k+1} + \sum_{q\neq j_k} U_qy_q^k,
\end{equation}
where the random index $j_k$ is chosen uniformly from the index set $J$.

\section{Convergence analysis}

The convergence analysis of the RDCIAG method is built on two pillars. The first is a tail vanishing lemma, which was proposed in \cite{Aytekin2016} for analyzing the linear convergence of IAG and recently developed to study PIAG and its variants. This result states that if the constant $c$ before the sum of tails $w_j$ with $j$ from $k-k_0$ to $k$ can be controlled by (less than) the constant $b$ before $w_k$, then the linear convergence rate $a$ can be conserved.
\begin{lemma}[\textbf{Pillar one: tail vanishing lemma}] \label{Aylem}
Assume that the nonnegative sequences $\{V_k\}$ and $\{w_k\}$ satisfy
$$V_{k+1}\leq a V_k -b w_k+ c\sum_{j=k-k_0}^k w_j,\quad \forall k\ge0,$$
where $a\in(0,1)$, $b\geq 0$, $c\geq 0$, and $k_0\ge0$. Assume also that $w_k=0$ for all $k<0$, and the following condition holds:
\begin{equation}\label{condition}
\frac{c}{1-a}\frac{1-a^{k_0+1}}{a^{k_0}}\leq b.
\end{equation}
Then $V_k\leq a^k V_0$ for all $k\geq 0$.
\end{lemma}

The second pillar is a generalized descent lemma. Almost all the convergence analysis of first-order methods is built on descent-type lemmas \cite{beck2009fast,beck2017first,Teboull2018A,zhang2016new}. In general, descent lemmas can be easily established by using the gradient-Lipschitz-continuous property and optimality conditions. However, in our case, more issues have to be taken into account to deal with the delayed terms and the random block coordinate updates. There are two ways to seek the required descent lemma. The first way is to take the delayed terms as an error term, as done by Bertesekas in \cite{Bertsekas2015incre,Gurbuzbalaban2015on}, and then follow the line of analysis for random block coordinate methods in \cite{Peter2014,Necoara2019Linear}. This way will leave the error term hard to copy with. The second way is more delicate. It first deduces a descent result based on the determined function values  $D(y^k)$ and  $D(\widetilde{y}^{k+1})$ conditioned on $(y^1,\cdots,y^k)$ and then establishes a descent lemma for $D(\widetilde{y}^{k+1})$ following the techniques in \cite{zhang2020}. As presented below, we succeed in the second way. All the missing proofs can be found in the appendix.

First, we show the gradient-Lipschitz-continuous property of $h_i(y)=f_i^*(-\sum_{j\in J}\cA_{ji}^*y_j)$.

\begin{proposition}\label{prop}
Assume that for all $i\in I$, $f_i$ is strongly convex with modulus $\mu_i$. Then for all $i\in I$, $\nabla h_i$ is Lipschtiz continuous with constant $\ell_i$, where
$$\ell_i:=\sqrt{(\sum_{j\in J}\frac{\|\cA_{ji}\|^2}{\mu_i^2}) |J|\max_{j\in J}\{\|\cA_{ji}^*\|^2 \}}.$$
\end{proposition}

The following is a descent result in terms of the determined function values $D(y^k)$ and $D(\widetilde{y}^{k+1})$, conditioned on $(y^1,\cdots,y^k)$.
\begin{proposition}\label{flem1}
Let $\eta_1:=\frac{(|J|-1)\sum_{i\in I}\ell_i}{|J|}$ and $\xi_k:=(y^1,\cdots,y^k)$ with $\ell_i$ defined as in Proposition \ref{prop}. Then we have
\begin{equation}\label{fina1}
\EE_{j_k}[D(y^{k+1})|\xi_k] \leq  \frac{|J|-1}{|J|}D(y^k)+\frac{1}{|J|}D(\widetilde{y}^{k+1})+ \frac{\eta_1}{|J|}\|\widetilde{y}^{k+1}-y^k\|^2.
\end{equation}
\end{proposition}

We are ready to present the required descent-type lemma for the convergence analysis.

\begin{lemma}[\textbf{Pillar two: descent-type lemma}] \label{flem2}
Let $\eta_1$  and $\xi_k$ be defined as in Proposition \ref{flem1} and let $\eta_2:=\frac{\ell_{\max}|I|(\tau+1)}{2}$ with $\ell_{\max}:=\max_{i\in I}\{\ell_i\}$. Then it follows that for all $y\in \cE_2$,
  \begin{eqnarray}\label{descent}
\EE_{j_k}[D(y^{k+1})|\xi_k]  & \leq  &\frac{|J|-1}{|J|}D(y^k)+\frac{1}{|J|}D(y)+ \left(\frac{\eta_1+\eta_2}{|J|}-\frac{1}{2\alpha |J|}\right)\|\widetilde{y}^{k+1}-y^k\|^2  \nonumber \\
  &&  +\frac{1}{2\alpha |J|} \|y-y^k\|^2 -\frac{1}{2\alpha|J|} \|y-\widetilde{y}^{k+1}\|^2+  \frac{\eta_2}{|J|}\sum_{s=k-\tau}^{k-1} \|y^{s+1}-y^s\|^2 .
 \end{eqnarray}
\end{lemma}

 \subsection{The convergence result}

In this subsection, we present the main convergence result for the RDCIAG method. The following result shows that the linear convergence of the PG method under error bound conditions could be kept even one uses delayed gradient information and randomly updates coordinate blocks.

\begin{theorem}\label{mainf2}
 Define the Lyapunov function
 $$\Gamma_\alpha(y):=D(y)-D^*+\frac{1}{2\alpha}d^2(y,\cY).$$
Assume that the bounded quadratic growth condition  with $(r,\sigma)$ in \eqref{QG} holds such that the iterate sequence $\{y^k\}\subseteq  \BB_r$.
If the stepsize $\alpha_k\equiv \alpha$ satisfies
 \begin{equation}\label{ss}
\alpha \leq \min\left\{\frac{z_0}{\sigma}, \frac{\eta_2}{8|J|}, \frac{1}{4(\eta_1+\eta_2)}\right\},
\end{equation}
where $z_0$ is the solution to the equation \eqref{inequ3} and $\eta_1,\eta_2$ are defined as in Lemma \ref{flem2}, then the following linear convergence results hold
 \begin{equation}\label{licon}
\EE[\Gamma_\alpha(y^k)]\leq  \left(1-\frac{\alpha\sigma}{|J|(1+\alpha\sigma)} \right)^k \Gamma_\alpha(y^0),
\end{equation}
and
 \begin{equation}\label{liconofx}
\EE[\|x^k-\bar{x}\|^2]\leq \left(2\alpha\Gamma_\alpha(y^0) \sum_{i\in I}\sum_{j\in J}\frac{\|\cA_{ji}\|^2}{\mu_i^2}\right)\left(1-\frac{\alpha\sigma}{|J|(1+\alpha\sigma)} \right)^{k-\tau}.
 \end{equation}
\end{theorem}
\begin{proof}
Let $\bar{y}^k:=\cP_\cY(y^k)$ be the projection of $y^k$ onto $\cY$. First,  noting that $\|\bar{y}^k-y^{k+1}\|^2\geq d^2(y^{k+1},\cY)$, we have
   \begin{eqnarray}\label{dist}
 \EE_{j_k}[d^2(y^{k+1}, \cY)|\xi_k]&\leq& \EE_{j_k}[ \|\bar{y}^k-y^{k+1}\|^2|\xi_k]\nonumber\\
 &=& \frac{1}{|J|}\sum_{j\in J}\|U_j\widetilde{y}_j^{k+1}+ \sum_{q\neq j}U_qy^k_q-\bar{y}^k\|^2\nonumber\\
 &= & \frac{1}{|J|}\sum_{j\in J} \left(\|\widetilde{y}_j^{k+1}-\bar{y}_j^k\|^2+ \sum_{q\neq j}\|y^k_q-\bar{y}_q^k\|^2\right)\nonumber\\
 &\leq &\frac{1}{|J|}\|\widetilde{y}^{k+1}-\bar{y}^k\|^2+\frac{|J|-1}{|J|}d^2(y^k,\cY).
 \end{eqnarray}
Using \eqref{descent} with $y=\bar{y}^k$ in Lemma \ref{flem2} and \eqref{dist}, and noting that $D(\bar{y}^k)=D^*$, we obtain
    \begin{eqnarray}\label{fina2}
  \EE_{j_k}[\Gamma_\alpha(y^{k+1})|\xi_k] &\leq&  \frac{|J|-1}{|J|}\Gamma_\alpha(y^k)+\frac{1}{2\alpha |J|}d^2(y^k,\cY)\nonumber\\
  &&-\left(\frac{1}{2\alpha}-\eta_1-\eta_2\right)\frac{1}{|J|}\|\widetilde{y}^{k+1}-y^k\|^2+ \frac{\eta_2}{|J|} \sum_{s=k-\tau}^{k-1} \|y^{s+1}-y^s\|^2.
 \end{eqnarray}
 By the bounded quadratic growth condition, it follows that for all $k$,
 \begin{equation}
 \frac{\sigma}{2} d^2(y^k,\cY)\leq D(y^k)-D^*.
  \end{equation}
Then one can verify that the following inequality holds
 \begin{equation}\label{grow}
 \frac{1}{2\alpha |J|}d^2(y^k,\cY) \leq \frac{1}{(1+\alpha\sigma)|J|}\Gamma_\alpha(y^k).
 \end{equation}
On the other hand, it is not hard to see that
 \begin{equation}\label{expec1}
 \EE_{j_k}[\|y^{k+1}-y^k\|^2|\xi_k]=\frac{1}{|J|}\|\widetilde{y}^{k+1}-y^k\|^2,
 \end{equation}
 which implies that
  \begin{equation}\label{expec2}
 \EE \|y^{k+1}-y^k\|^2=\EE_{\xi_k}\EE_{j_k} [\|y^{k+1}-y^k\|^2|\xi_k]=\frac{1}{|J|}\EE_{\xi_k}\|\widetilde{y}^{k+1}-y^k\|^2.
 \end{equation}
From \eqref{grow} and \eqref{expec2}, taking expectation with respect to $\xi_k$ on \eqref{fina2} implies
    \begin{eqnarray}\label{fina3}
  \EE \Gamma_\alpha(y^{k+1}) &\leq&  \left(1-\frac{\alpha\sigma}{|J|(1+\alpha\sigma)}\right)\EE\Gamma_\alpha(y^k) \nonumber\\
  &&-\left(\frac{1}{2\alpha}-\eta_1-\eta_2\right)  \EE \|y^{k+1}-y^k\|^2 + \frac{\eta_2}{|J|} \sum_{s=k-\tau}^{k-1} \EE\|y^{s+1}-y^s\|^2.
 \end{eqnarray}
By the choice of $\alpha$,  it follows that $\alpha \leq \frac{1}{4(\eta_1+\eta_2)}$. Then it follows from \eqref{fina3} that
      \begin{eqnarray}\label{fina4}
  \EE \Gamma_\alpha(y^{k+1}) &\leq&  \left(1-\frac{\alpha\sigma}{|J|(1+\alpha\sigma)}\right)\EE\Gamma_\alpha(y^k)  \nonumber\\
  &&- \frac{1}{4\alpha} \EE \|y^{k+1}-y^k\|^2 + \frac{\eta_2}{|J|} \sum_{s=k-\tau}^k \EE\|y^{s+1}-y^s\|^2.
 \end{eqnarray}
Let $V_k:=\EE \Gamma_\alpha(y^k)$, $w_k:=\EE\|y^{k+1}-y^k\|^2$,  $a:=1-\frac{\alpha\sigma}{|J|(1+\alpha\sigma)}$, $b:=\frac{1}{4\alpha}$,  $c=\frac{\eta_2}{|J|}$. Then \eqref{fina4} becomes
  \begin{eqnarray}\label{fina5}
  V_{k+1} \leq  a V_k -b w_k + c \sum_{s=k-\tau}^{k} w_s.
 \end{eqnarray}
To employ Lemma \ref{Aylem}, it remains to determine the stepsize $\alpha$ such that the following condition holds
$$\frac{c}{1-a}\frac{1-a^{\tau+1}}{a^{\tau}}\leq b.$$
Let $\beta:=1-\frac{1}{|J|}$. Then $a=\frac{1+\beta \alpha \sigma}{1+\alpha \sigma}$. After some simple calculations, the condition above becomes
$$\frac{1}{a^\tau}\leq 1+\frac{\sigma(1-\beta)}{1+\alpha\sigma}(\frac{c}{4}-\alpha).$$
Since $\alpha\leq \frac{c}{8}$ by the choice of $\alpha$, it suffices to require that
\begin{equation}\label{inequ1}
\frac{1}{a^\tau}=\left(\frac{1+\alpha\sigma}{1+\beta\alpha\sigma}\right)^\tau\leq 1+\frac{c\sigma(1-\beta)}{8(1+\alpha\sigma)}.
\end{equation}
Denote $\gamma:=\frac{c\sigma(1-\beta)}{8}$ and $z:=\alpha \sigma$. Then \eqref{inequ1} becomes
\begin{equation}\label{inequ2}
\left(\frac{1+z}{1+\beta z}\right)^\tau \leq 1+\frac{\gamma}{1+z}.
\end{equation}
It is not hard to find that there exists $z_0$ such that when $0<z\leq z_0$, or equivalently $\alpha \leq \frac{z_0}{\sigma}$, \eqref{inequ2} always holds. Actually, by the monotonicity of $\left(\frac{1+z}{1+\beta z}\right)^\tau$ and $1+\frac{\gamma}{1+z}$ with respect to $z$, we can choose $z_0$ as the solution to the equation
\begin{equation}\label{inequ3}
\left(\frac{1+z}{1+\beta z}\right)^\tau = 1+\frac{\gamma}{1+z}.
\end{equation}
Collecting all the bounds on $\alpha$, we can conclude that if the stepsize $\alpha$ satisfies \eqref{ss},
then the linear convergence result \eqref{licon} follows by Lemma \ref{Aylem}.

It remains to show the convergence result \eqref{liconofx}. Using the expression of $x^k$ in \eqref{alg3x} and the Lipschtiz continuity property of $\nabla f^*_i$, and letting $\bar{y}^k:=\cP_{\cY}(y^{k-\tau_k^i})$ for a fixed index $i\in I$,  we derive that
\begin{eqnarray*}
  \|x^k_i-\bar{x}_i\| &=&  \|\nabla f_i^*( -\sum_{j\in J}\cA_{ji}^*y_j^{k-\tau_k^i})- \nabla f_i^*( -\sum_{j\in J}\cA_{ji}^*\bar{y}_j^k)\|  \nonumber\\
  &\leq &\frac{1}{\mu_i}\sum_{j\in J}\|\cA_{ji}\|\|\bar{y}^k_j-y_j^{k-\tau_k^i}\|   \nonumber\\
  &\leq &\frac{\sqrt{\sum_{j\in J}\|\cA_{ji}\|^2}}{\mu_i}\sqrt{\sum_{j\in J} \|\bar{y}_j^k-y_j^{k-\tau_k^i}\|^2}
  \nonumber\\
  &\leq &\frac{\sqrt{\sum_{j\in J}\|\cA_{ji}\|^2}}{\mu_i}  d(y^{k-\tau_k^i},\cY),
 \end{eqnarray*}
where the second inequality follows from the Cauchy inequality. Thus
\begin{equation}\label{ine}
\|x^k-\bar{x}\|^2=\sum_{i\in I}\|x^k_i-\bar{x}_i\| ^2\leq \sum_{i\in I}\sum_{j\in J}\frac{\|\cA_{ji}\|^2}{\mu_i^2}d^2(y^{k-\tau_k^i},\cY).
\end{equation}
Using  \eqref{licon}, we have
\begin{eqnarray}\label{ine2}
  \EE[d^2(y^{k-\tau_k^i},\cY)] &\leq &  2\alpha \EE[\Gamma_\alpha(y^{k-\tau_k^i})] \nonumber\\
  &\leq & 2\alpha \left(1-\frac{\alpha\sigma}{|J|(1+\alpha\sigma)} \right)^{k-\tau} \Gamma_\alpha(y^0),
 \end{eqnarray}
 where we use the fact that $\tau_k^i\leq \tau$ for all $k$ and $i$. Then \eqref{liconofx} follows from \eqref{ine} and \eqref{ine2} immediately. The proof is complete.
\end{proof}

\section{Application examples}
In this section, we present three application examples to illustrate the prospect of our proposed method.

\subsection{Best approximation problem}
As the first application example, we consider the best approximation problem, i.e., finding the best approximation to a given point $v$ from the intersection of some closed convex sets $\Omega_0, \Omega_i, i\in I$. Mathematically, we solve the following minimization problem
\begin{eqnarray}\label{app1}
\begin{array}{ll}
\Min & \frac{1}{2}\|x-v\|^2\\[4pt]
\st & x\in \bigcap_{i=1}^m \Omega_i, x\in \Omega_0.
\end{array}
\end{eqnarray}
Among the iterative algorithms for solving this problem, the (random) Dykstra method is one of the first projection-based algorithms, whose linear convergence was recently established under very mild assumptions; see e.g. \cite{Necoara2019Linear}. If each iteration point $x^k$ is required to lie in $\Omega_0$, the Dykstra-type methods will not be applicable. Note that the requirement about that $x^k$ lies in $\Omega_0$ is a very natural constraint in many practical cases; e.g. in imaging processing the pixel of denosing/deblurring imagines has to belong to some certain interval. Interestingly, our proposed algorithm could meet this requirement. To this end, we denote
\begin{equation}\label{fx}
f(x):= \frac{1}{2}\|x-v\|^2+\delta_{\Omega_0}(x)
\end{equation}
and $g_j(x):=\delta_{\Omega_j}(x)$ for all $j=1,\cdots,m$. Then problem \eqref{app1} can be reformulated as
$$\Min f(x)+\sum_{j=1}^m g_j(x),$$
so that the proposed RDCIAG method can be applied.
Using the expressions of $f(x)$ and $g_j(x)$, it is not hard to verify that
$$\nabla f^*(y)=\cP_{\Omega_0}(v+y)$$ and
$$\prox_{\alpha g^*_j}(y)=y-\alpha \cP_{\Omega_j}(\alpha^{-1}y).$$
Thus the iterative scheme of applying the proposed RDCIAG method to solve problem \eqref{app1} is
\begin{eqnarray}\label{alg3}
\left\{\begin{array}{lll}
x^k &= &\cP_{\Omega_0}(v-\sum_{j\in J}y_j^k),   \\[5pt]
y_j^{k+1}&=& y_j^k+\alpha_k x^k-\alpha_k \cP_{\Omega_j}(\alpha_k^{-1}y_j^k+x^k),\quad j\in J.
\end{array} \right.
\end{eqnarray}
Theoretically, if the constraints $\Omega_i$ are polyhedral, then both $f$ and $g_j$ are piecewise linear-quadratic and their subdifferentials are polyhedral. Thus the bounded quadratic growth condition holds by Proposition \ref{prop-1} and the iterative scheme \eqref{alg3} converges linearly by Theorem \ref{mainf2}.
However, since $f$ is nonsmooth, the theoretical results proposed in  \cite{Necoara2019Linear} cannot be applied to analyze the algorithm above.

\subsection{Sparse optimization problem}
In this subsection, we point out that our proposed algorithm is suitable for solving the augmented $\ell_1$ minimization problem
  \begin{eqnarray}\label{app2}
\begin{array}{ll}
\Min & \lambda \|x\|_1+\frac{1}{2}\|x\|^2\\[4pt]
\st & Ax=b,
\end{array}
\end{eqnarray}
where $A\in \RR^{m\times n}$ is a given matrix with rows $a_i$, $b\in \RR^m$ is a given vector with entries $b_i$, and $\lambda>0$ is a regularization parameter. In compressive sensing, the matrix $A$ represents the compressed linear measure and hence the number $m$ of measures is much less than the dimension $n$ of the signal $x$. The well-known algorithm for solving this problem is the linearized Bregman method \cite{Lai2012Augmented}, which is actually the dual gradient descent applied to the Lagrangian dual problem to \eqref{app2}. Recently, the authors of \cite{Lorenz2014,Frank2019linear} proposed (randomized) sparse Kaczmarz algorithms by viewing the linearized Bregman method as a Bregman projection method. If we let $f(x):= \lambda \|x\|_1+\frac{1}{2}\|x\|^2$, then the randomized sparse Kaczmarz algorithm reads as
  \begin{eqnarray}\label{SKA}
\left\{\begin{array}{lll}
x^{k+1}_*&=&x^k_*-\frac{\langle a_i, x^k\rangle-b_i}{\|a_i\|^2}\cdot a_i,    \\
x^{k+1}&=& \nabla f^*(x^{k+1}_*),
\end{array} \right.
\end{eqnarray}
where the index $i\in \{1,\cdots, m\}$ is chosen randomly. It was observed (see \cite{Frank2019linear} and its reference) that the randomized sparse Kaczmarz algorithm could be identified as a random dual coordinate descent method applied to the dual objective function
\begin{equation}\label{df1}
\frac{1}{2}\|\nabla f^*(A^Ty)\|^2-\langle b, y\rangle.
\end{equation}
In these dual-type methods only the case of large $m$ is exploited via choosing the index $i\in \{1,\cdots, m\}$ randomly. However, as pointed out previously, $n$ is much larger than $m$ in compressive sensing. Therefore, our proposed algorithm could be applicable to the case of large $m$ and $n$. We now describe this case more clearly. First, we view the constraint $Ax=b$ as an intersection of the hyperplanes $\Omega_j:=\{x: \langle a_j, x\rangle =b_j\}, j = 1,\cdots, m$ and write down $f(x)=\sum_{i=1}^n f_i(x_i)$ with $f_i(x_i)=\lambda|x_i|+\frac{1}{2}|x_i|^2$. From this point of view, problem \eqref{app2} can be written as
\begin{equation}\label{BP2}
\Min \sum_{i=1}^n f_i(x_i)+ \sum_{i=1}^m \delta_{\Omega_i}(x).
\end{equation}
Its Fenchel-Rockafellar dual problem reads as
\begin{equation*}
\Min_{y_j\in \RR^n,j=1,\cdots,m}  \sum_{i=1}^n f_i^*(- \sum_{j=1}^m y_{ji}) +  \sum_{j=1}^m\delta_{\Omega_j}^*(y_j),
\end{equation*}
whose objective function is obviously different from \eqref{df1}.
Since the objective function is piecewise linear-quadratic and the constraint are polyhedral convex, it then follows that $\partial f_i$ and $\partial \delta_{\Omega_j}$ are polyhedral. Thus, the bounded quadratic growth conditions holds by Proposition \ref{prop-1}. Applying the proposed algorithm to solve problem \eqref{BP2} can get an optimal solution linearly by Theorem \ref{mainf2}.

\subsection{Network Utility Maximization}
As the last example, we revisit the network utility maximization problem that was discussed in Beck's book \cite{beck2017first}. To recover the existing algorithm, we follow the description of this problem in \cite{beck2017first}.  Consider a network that consists of a set $\cS=\{1,2,\cdots, S\}$ of sources and a set $\cL=\{1,2,\cdots, L\}$ of links, where a link $\ell$ has a capacity $c_\ell$. For each source $s\in \cS$, the set of all links used by source $s$ is denoted by $\cL(s)\subseteq \cL$. For a given link $\ell\in\cL$, the set of all sources that use link $\ell$ is denoted by $\cS(\ell)\subseteq \cS$. Then, $\cS(\ell)$ and $\cL(s)$ have the relation that $s\in \cS(\ell)$ if and only if $\ell\in\cL(s)$. Each source $s\in \cS$ is associated with a concave utility function $u_s$, meaning that if source $s$ sends data at a rate $x_s$, it gains a utility $u_s(x_s)$. Assume that the rate $x_s$ lies in the interval $I_s=[0, M_s]$ with $M_s>0$ being given. The network utility maximization problem is to allocate optimally the source rates, which is mathematically formulated as
\begin{eqnarray}\label{app3}
\begin{array}{lll}
\Max & \sum_{s\in\cS}u_s(x_s) &\\[4pt]
\st & \sum_{s\in \cS(\ell)}x_s\leq c_\ell, & \ell \in \cL,\\ [4pt]
& x_s\in I_s, & s\in \cS.
\end{array}
\end{eqnarray}
In order to reformulate it as a special case of our framework, we let $I=\cS, J=\cL$, $f_s(x_s)=-u_s(x_s)+\delta_{I_s}(x_s)+\frac{\lambda}{2}\|x_s\|^2,  g_\ell(\cdot)=\delta_{(-\infty, c_\ell]}(\cdot)$ with $\lambda\geq0$ being a regularization parameter. Define
 \begin{eqnarray*}
A_{\ell s}:=\left\{\begin{array}{ll}
1, & s\in \cS(\ell),\\
0, & \textrm{otherwise},
\end{array} \right.
\end{eqnarray*}
and $\cA_\ell x:= \sum_{s\in \cS}A_{\ell s}x_s=\sum_{s\in \cS(\ell)}x_s$. Then the problem above can be written as
 \begin{equation}\label{app3.1}
\Min \sum_{i\in I} f_i(x_i)+\sum_{j\in J} g_j(\cA_jx).
\end{equation}
Applying the iterate \eqref{alg2} to solve problem \eqref{app3.1} yields
 \begin{eqnarray}\label{app3.2}
\left\{\begin{array}{lll}
x_i^k &= & \nabla f_i^*( -\sum_{j\in J}\cA_{ji}^*y_j^{k-\tau_k^i}),\quad i\in I,   \\[5pt]
y_j^{k+1}&=& \prox_{\alpha_kg^*_j}(y_j^k+\alpha_k\sum_{i\in I}\cA_{ji}x^k_i),\quad j_k\in J,
\end{array} \right.
\end{eqnarray}
where $j_k\in J$ is uniformly chosen.
Note that
$$\nabla f_i^*(y)=\arg\min_x\{\langle x, y\rangle-f_i(x)\},$$
and
$$\prox_{\alpha_kg^*_j}(y)=[y-\alpha_k c_j]_+.$$
Since $\sum_{j\in J}\cA_{ji}^*y_j^{k-\tau_k^i}=\sum_{j\in \cL(i)}y_j^{k-\tau_k^i}$, using the notations involved in the original problem, we can rewrite  the algorithm \eqref{app3.2} as
 \begin{eqnarray}\label{app3.3}
\left\{\begin{array}{lll}
x_s^k &= & \arg\min_{x_s\in I_s}\left\{\frac{\lambda}{2}\|x_s\|^2-u_s(x_s)+ (\sum_{\ell\in \cL (s)}y_\ell^{k-\tau_k^s}) x_s\right\},\quad s\in \cS, \\[4pt]
y_\ell^{k+1}&=& [y_\ell^k+\alpha_k \sum_{s\in \cS(\ell)} x^k_s-\alpha_k c_j]_{+},\quad \ell \in \cL,
\end{array} \right.
\end{eqnarray}
where $\ell \in \cL$ is uniformly chosen at random. If $\lambda=0$ and $\tau_k^s\equiv 0$, and updating $y_\ell^k$ for all $\ell\in \cL$, then the iterate \eqref{app3.3} reduces to the dual projected subgradient method presented in \cite{beck2017first}.

\section{Conclusions}
In this paper, we have proposed a hybrid algorithm by blending the well-known random dual block coordinate descent method and the recently popularized PIAG method to deal with a class of large-scale problems. Based on a newly established generalized descent lemma, the linear convergence of the proposed algorithm is derived under the bounded error bound condition. Finally, some application examples have been illustrated by modifying or extending several existing algorithms. The proposed algorithm may be accelerated with the help of the restart or inertial accelerated techniques as employed in \cite{Necoara2019Linear}. Moreover, we would like to generalize the proposed algorithm in non-Euclidean spaces and enhance the linear convergence by using H$\ddot{o}$lderian error bound conditions. We leave them for future research.

\section*{Acknowledgements}
This work is supported by the National Science Foundation of China (Nos.11971480, 11771287, 71632007, 11631013, and 11991020), the Beijing Academy of Artificial Intelligence, and the Natural Science Fund of Hunan for Excellent Youth (No.2020JJ3038).

\section{Appendix: The missing proofs}

\noindent {\bf The proof of Proposition \ref{pro-bms}}:
Let $y\in \partial \psi^{-1}(x)$, which is equivalent to $x\in \partial \psi(y)$. By \eqref{prop1}, we have
\[
d (y, \partial \psi^{-1}(\bar{x})) \le \kappa \|x-\bar{x}\|,\quad \forall x\in \BB(\bar{x},\delta) \cap \partial \psi(y).
\]
This means that
\[
d (y, \partial \psi^{-1}(\bar{x})) \le \kappa d (\bar{x},\BB(\bar{x},\delta) \cap \partial \psi(y)).
\]
For any $y$ such that $d(\bar{x},\partial \psi(y))\le \delta$, the above inequality implies
\[
d (y, \partial \psi^{-1}(\bar{x})) \le \kappa d (\bar{x}, \partial \psi(y)).
\]
Then by \cite[Proposition 1]{Ye2018}, for any $r>0$ there exists $k_r>0$ such that
\[
d (y, \partial \psi^{-1}(\bar{x})) \le \kappa_r d (\bar{x}, \partial \psi(y)),\quad \forall y\in \BB_r.
\]
The proof is complete.

\bigskip

\noindent {\bf The proof of Proposition \ref{prop-1}}:
(i) Applying \cite[Theorem 4.3]{Drusvyatskiy2018Error} to the dual problem \eqref{d3} requires three conditions. The first is that $f^*$ is essentially smooth and the primal problem has a unique minimizer $\bar{x}$, which can be satisfied if $f$ is essentially strictly convex. The second is the BLR property for the pair of sets $\partial g(\cA\bar{x})$ and $(\cA^*)^{-1}(-\partial f(\bar{y}))$ to upper estimate $d(y,\cY)$, which can be satisfied by condition \eqref{cond2}.

The third is the firm convexity of the sum functions $f^*$ and $g^*$. Indeed, since $\partial f_i$ is upper Lipschitz continuous at $\bar{x}_i$, by Proposition \ref{pro-bms} $\partial f^*_i$ is BMS at $\bar{x}_i$ and then $f_i^*$ is firmly convex with respect to $\bar{x}_i$. In the same way, it follows that $g_j^*$ is firmly convex with respect to $\cA_j \bar{x}$. Then by \cite[Lemma 5]{Drusvyatskiy2018Error}, the sum functions $f^*$ and $g^*$ are firmly convex with respect to $\bar{x}$ and $\cA \bar{x}$ respectively.

Choosing a compact set $\BB_r$ to replace the sublevel set $\{y: D(y)\le D^*+v\}$ used in the proof of \cite[Theorem 4.3]{Drusvyatskiy2018Error}, the desired result follows immediately.

(ii) Noting that for a closed proper convex function $\psi$, $\partial \psi$ is polyhedral if and only if $\partial \psi^*$ is polyhedral, the desired result follows immediately from \cite[Corollary 4.4]{Drusvyatskiy2018Error}.

\bigskip

\noindent {\bf The proof of Proposition \ref{prop}}:
Recall that $h_i(y)=f_i^*(-\sum_{j\in J}\cA_{ji}^*y_j)$. We have
$$\nabla h_i(y)=(\nabla_j h_i(y))_{j\in J}=\left[(-\cA_{ji})\nabla f^*_i(-\sum_{j\in J}\cA_{ji}^*y_j)\right]_{j\in J}.$$
Since $f_i$ is strongly convex with modulus $\mu_i$, $\nabla f_i^*$ must be $\frac{1}{\mu_i}$-Lipschitz continuous. Thereby, we derive that
  \begin{eqnarray}\label{Lip1}
\begin{array}{ll}
\|\nabla h_i(y)-\nabla h_i( y^\prime)\|^2 & =  \sum_{j\in J} \| \nabla_j h_i(y)- \nabla_j h_i(y^\prime)\|^2 \\
 &=  \sum_{j\in J}  \| \cA_{ji} \nabla f^*_i(-\sum_{j\in J}\cA_{ji}^*y^\prime_j) - \cA_{ji} \nabla f^*_i(-\sum_{j\in J}\cA_{ji}^*y_j)\|^2\\
  &\leq  \sum_{j\in J}\|\cA_{ji}\|^2  \|  \nabla f^*_i(-\sum_{j\in J}\cA_{ji}^*y^\prime_j) -\nabla f^*_i(-\sum_{j\in J}\cA_{ji}^*y_j)\|^2\\
    &\leq  (\sum_{j\in J}\frac{\|\cA_{ji}\|^2}{\mu_i^2})  \| \sum_{j\in J}\cA_{ji}^*y_j-\sum_{j\in J}\cA_{ji}^*y^\prime_j \|^2\\
    &\leq  (\sum_{j\in J}\frac{\|\cA_{ji}\|^2}{\mu_i^2}) |J| \sum_{j\in J}\|\cA_{ji}^*\|^2 \|y_j-y^\prime_j \|^2\\
       &\leq  (\sum_{j\in J}\frac{\|\cA_{ji}\|^2}{\mu_i^2}) |J|\max_{j\in J}\{\|\cA_{ji}^*\|^2 \} \sum_{j\in J}\|y_j-y^\prime_j \|^2\\
        &=  (\sum_{j\in J}\frac{\|\cA_{ji}\|^2}{\mu_i^2}) |J|\max_{j\in J}\{\|\cA_{ji}^*\|^2 \}  \|y -y^\prime \|^2,
\end{array}
\end{eqnarray}
where the third inequality follows from the Jensen inequality.

\bigskip

\noindent {\bf The proof of Proposition \ref{flem1}:}  We divide the proof into four steps.

\textbf{Step 1.} We first show that
\begin{equation}\label{bound1}
\EE_{j_k}\Big[\sum_{j\in J}g^*_j(y_j^{k+1})|\xi_k\Big]= \frac{|J|-1}{|J|}\sum_{j\in J}g^*_j(y_j^k)+\frac{1}{|J|}\sum_{j\in J}g^*_j(\widetilde{y}_j^{k+1}).
\end{equation}
Actually, denoting $G(y):= \sum_{j\in J}g^*_j(y_j)$ and taking the conditional expectation over $j_k$ conditioned on $\xi_k$, we can derive that
\begin{eqnarray}\label{EEbound1}
\EE_{j_k}\Big[\sum_{j\in J}g^*_j(y_j^{k+1})|\xi_k\Big]& = & \EE_{j_k}[G(y^{k+1})|\xi_k] \nonumber \\
&=& \EE_{j_k}\Big[G(U_{j_k}\widetilde{y}_{j_k}^{k+1} + \sum_{q\neq j_k} U_qy_q^k)|\xi_k\Big] \nonumber \\
&=& \frac{1}{|J|}\sum_{j\in J} G(U_j\widetilde{y}_j^{k+1} + \sum_{q\neq j} U_qy_q^k)  \nonumber \\
&=& \frac{1}{|J|}\sum_{j\in J} \left(g^*_j(\widetilde{y}_j^{k+1})+ \sum_{q\neq j}g^*_q(y_q^k)\right) \nonumber \\
&=& \frac{|J|-1}{|J|}\sum_{j\in J}g^*_j(y_j^k)+\frac{1}{|J|}\sum_{j\in J}g^*_j(\widetilde{y}_j^{k+1}).
\end{eqnarray}

\textbf{Step 2.} We show that
\begin{equation}\label{bound2}
\EE_{j_k}\Big[\sum_{i\in I}h_i(y^{k+1})|\xi_k\Big]\leq  \sum_{i\in I}\left( h_i(y^k)+\frac{1}{|J|}\langle \nabla h_i(y^k), \widetilde{y}^{k+1}-y^k
\rangle +\frac{\ell_i}{2|J|}\|\widetilde{y}^{k+1}-y^k\|^2\right).
\end{equation}
By Proposition \ref{prop}, we derive that
\begin{eqnarray}
h_i(y^{k+1})& = & h_i(U_{j_k}\widetilde{y}_{j_k}^{k+1} + \sum_{q\neq j_k} U_qy_q^k) \nonumber \\
&\leq & h_i(y^k)+\langle \nabla h_i(y^k), U_{j_k}(\widetilde{y}_{j_k}^{k+1}-y^k_{j_k} )\rangle +\frac{\ell_i}{2}\|U_{j_k}(\widetilde{y}_{j_k}^{k+1}-y^k_{j_k} )\|^2. \nonumber
\end{eqnarray}
Taking the conditional expectation over $j_k$ conditioned on $\xi_k$, we obtain
\begin{eqnarray*}
 \EE_{j_k}[h_i(y^{k+1})|\xi_k]&\leq&  h_i(y^k)+ \frac{1}{|J|}\sum_{j\in J}\langle \nabla h_i(y^k), U_{j}(\widetilde{y}_j^{k+1}-y^k_j )\rangle +\frac{1}{|J|}\sum_{j\in J} \frac{\ell_i}{2}\|U_j(\widetilde{y}_j^{k+1}-y^k_j )\|^2 \nonumber \\
 &\leq&  h_i(y^k)+ \frac{1}{|J|} \langle \nabla h_i(y^k), \widetilde{y}^{k+1}-y^k\rangle +\frac{1}{|J|} \frac{\ell_i}{2}\| \widetilde{y}^{k+1}-y^k\|^2,
\end{eqnarray*}
from which \eqref{bound2} follows.

\textbf{Step 3.} We show that
\begin{equation}\label{bound2.1}
\EE_{j_k}[\sum_{i\in I}h_i(y^{k+1})|\xi_k]\leq  \sum_{i\in I}\left( h_i(\widetilde{y}^{k+1})+ \frac{|J|-1}{|J|} \langle \nabla h_i(\widetilde{y}^{k+1})), y^k-\widetilde{y}^{k+1}\rangle +\frac{|J|-1}{|J|} \frac{\ell_i}{2}\|y^k-\widetilde{y}^{k+1}\|^2\right).
\end{equation}
By Proposition \ref{prop}, we derive that
\begin{eqnarray}
h_i(y^{k+1})& = & h_i(U_{j_k}\widetilde{y}_{j_k}^{k+1} + \sum_{q\neq j_k} U_qy_q^k) \nonumber \\
&\leq & h_i(\widetilde{y}^{k+1} )+\langle \nabla h_i(\widetilde{y}^{k+1}), \sum_{q\neq j_k}U_q(y^k_q-\widetilde{y}_q^{k+1} )\rangle +\frac{\ell_i}{2}\|\sum_{q\neq j_k}U_q(y^k_q-\widetilde{y}_q^{k+1} )\|^2. \nonumber
\end{eqnarray}
Taking the conditional expectation over $j_k$ conditioned on $\xi_k$, we obtain
\begin{eqnarray}
 \EE_{j_k}[h_i(y^{k+1})|\xi_k]&\leq&  h_i(\widetilde{y}^{k+1}))+ \frac{1}{|J|}\sum_{j\in J}\langle \nabla h_i(\widetilde{y}^{k+1}), \sum_{q\neq j}U_q(y^k_q-\widetilde{y}_q^{k+1} )\rangle  +\frac{1}{|J|}\sum_{j\in J} \frac{\ell_i}{2}\|\sum_{q\neq j}U_q(y^k_q-\widetilde{y}_q^{k+1} )\|^2 \nonumber \\
 &\leq&  h_i(\widetilde{y}^{k+1}))+ \frac{|J|-1}{|J|} \langle \nabla h_i(\widetilde{y}^{k+1})), y^k-\widetilde{y}^{k+1}\rangle +\frac{|J|-1}{|J|} \frac{\ell_i}{2}\|y^k-\widetilde{y}^{k+1}\|^2,
\end{eqnarray}
from which \eqref{bound2.1} follows.

\textbf{Step 4.} From $\eqref{bound2}\times \frac{|J|-1}{|J|}+\eqref{bound2.1}\times \frac{1}{|J|}$ and using the monotonicity of $\nabla h_i$, we derive that
\begin{eqnarray}\label{bound2.2}
\EE_{j_k}[\sum_{i\in I}h_i(y^{k+1})|\xi_k] &\leq & \frac{|J|-1}{|J|}\sum_{i\in I}  h_i(y^k)+ \frac{1}{|J|}\sum_{i\in I}  h_i(\widetilde{y}^{k+1}) + \sum_{i\in I}  \frac{(|J|-1)\ell_i}{|J|^2}\|y^k-\widetilde{y}^{k+1}\|^2 \nonumber \\
 &&  +\frac{|J|-1}{|J|^2}\sum_{i\in I} \langle  \nabla h_i(y^k)-\nabla h_i(\widetilde{y}^{k+1})), \widetilde{y}^{k+1}-y^k\rangle
 \nonumber \\
 &\leq & \frac{|J|-1}{|J|}\sum_{i\in I}  h_i(y^k)+ \frac{1}{|J|}\sum_{i\in I}  h_i(\widetilde{y}^{k+1}) + \sum_{i\in I}  \frac{(|J|-1)\ell_i}{|J|^2}\|y^k-\widetilde{y}^{k+1}\|^2.
\end{eqnarray}
The proof is complete by summing \eqref{EEbound1} and \eqref{bound2.2}.

\bigskip

\noindent {\bf The proof of Lemma \ref{flem2}:}
Since $h_i(x)$ is convex and gradient-Lipschitz-continuous by Proposition \ref{prop}, it follows that
\begin{align}\label{Lp}
h_i(\widetilde{y}^{k+1})\leq &h_i(y^{k-\tau_k^i})+\langle \nabla h_i(y^{k-\tau_k^i}), \widetilde{y}^{k+1}-y^{k-\tau_k^i}\rangle + \frac{\ell_i}{2}  \|\widetilde{y}^{k+1}-y^{k-\tau_k^i}\|^2 \nonumber\\
\leq  &h_i(y)+\langle \nabla h_i(y^{k-\tau_k^i}), \widetilde{y}^{k+1}-y\rangle + \frac{\ell_i}{2}  \|\widetilde{y}^{k+1}-y^{k-\tau_k^i}\|^2.
\end{align}
Note that $\tau_k^i\leq \tau$ and  $\ell_{\max}=\max_{i\in I}\{\ell_i\}$. Using the following inequality
$$
\|v_{k}-v_j\|^2=\|\sum_{i=j}^{k-1}(v_{i+1}-v_i)\|^2\leq (k-j)\sum_{i=j}^{k-1}\|v_{i+1}-v_i\|^2,\quad \forall  k>j\ge1.
$$
and summing \eqref{Lp} over all $i\in I$, we obtain
\begin{eqnarray}\label{obj1}
\sum_{i\in I}h_i(\widetilde{y}^{k+1})& \leq & \sum_{i\in I}h_i(y)+\langle \sum_{i\in I} \nabla h_i(y^{k-\tau_k^i}), \widetilde{y}^{k+1}-y\rangle \nonumber \\
&& +\frac{\ell_{\max}|I|(\tau+1)}{2}\left(\|\widetilde{y}^{k+1}-y^k\|^2+\sum_{s=k-\tau}^{k-1} \|y^{s+1}-y^s\|^2\right).
\end{eqnarray}
By the optimality of $\widetilde{y}^{k+1}_j$ in \eqref{update1} and Fermat's rule, we have
\begin{equation}\label{m01}
0\in \sum_{i\in I} \nabla_{j} h_i(y^{k-\tau_k^i})+\alpha^{-1} (\widetilde{y}^{k+1}_j-y^k_{j})+\partial g^*_{j}(\widetilde{y}^{k+1}_j).
\end{equation}
This together with the subgradient inequality for the convex function $g^*_j(y_j)$ at $\widetilde{y}^{k+1}_j$ implies that
  \begin{eqnarray}\label{obj2}
g^*_j(\widetilde{y}^{k+1}_j)&\leq & g^*_j(y_j) + \langle \sum_{i\in I} \nabla_j h_i(y^{k-\tau_k^i})+\alpha^{-1}(\widetilde{y}^{k+1}_j-y^k_j), y_j-\widetilde{y}^{k+1}_j\rangle
\end{eqnarray}
 Thus, we obtain
 \begin{equation}\label{obj2.1}
 \sum_{j\in J}g^*_{j}(\widetilde{y}_{j}^{k+1})  \leq \sum_{j\in J}g^*_{j}(y_j) + \langle \sum_{i\in I} \nabla h_i(y^{k-\tau_k^i}),y -\widetilde{y}^{k+1}  \rangle + \alpha^{-1} \langle \widetilde{y}^{k+1} -y^k , y -\widetilde{y}^{k+1} \rangle.
 \end{equation}
 Noting $\eta_2=\frac{\ell_{\max}|I|(\tau+1)}{2}$ and summing up \eqref{obj1} and \eqref{obj2.1}, we obtain
  \begin{equation}\label{obj}
 D(\widetilde{y}^{k+1}) \leq D(y) + \frac{1}{\alpha} \langle \widetilde{y}^{k+1} -y^k , y -\widetilde{y}^{k+1} \rangle + \eta_2 (\|\widetilde{y}^{k+1}-y^k\|^2+\sum_{s=k-\tau}^{k-1} \|y^{s+1}-y^s\|^2).
 \end{equation}
 Note that
 $$\langle \widetilde{y}^{k+1} -y^k , y -\widetilde{y}^{k+1} \rangle=\frac{1}{2}\|y-y^k\|^2-\frac{1}{2}\|y-\widetilde{y}^{k+1}\|^2-\frac{1}{2}\|\widetilde{y}^{k+1}-y^k\|^2.$$
Thereby, we have
  \begin{eqnarray}\label{obj3}
 D(\widetilde{y}^{k+1}) & \leq  &D(y) + \frac{1}{2\alpha} \|y-y^k\|^2 -\frac{1}{2\alpha} \|y-\widetilde{y}^{k+1}\|^2-\frac{1}{2\alpha} \|\widetilde{y}^{k+1}-y^k\|^2 \nonumber \\
  &&  +\eta_2 \left(\|\widetilde{y}^{k+1}-y^k\|^2+\sum_{s=k-\tau}^{k-1} \|y^{s+1}-y^s\|^2\right),\quad \forall y\in \cE_2.
 \end{eqnarray}
The descent lemma follows by combining \eqref{fina1} and \eqref{obj3}. This completes the proof.

\bibliographystyle{abbrv}

\end{document}